\documentclass[10pt,reqno]{amsart}
\usepackage{amsmath,amscd,amsthm,a4wide,amssymb,amsxtra,mathrsfs,amsrefs}

\allowdisplaybreaks

\theoremstyle{plain}
\newtheorem{thm}{Theorem}[section]

\newtheorem{cor}[thm]{Corollary}

\theoremstyle{definition}
\newtheorem{rem}[thm]{Remark}

\newtheorem{defi}[thm]{Definition}

\newtheorem{conv}[thm]{Convention}

\numberwithin{equation}{section}

\def\loc{\operatorname{loc}}

\def\esup{\operatornamewithlimits{ess\,sup}}

\def\Id{\operatorname{I}}

\def\ap{\approx}
\def\qq{\qquad}
\def\rw{\rightarrow}

\def\ls{\lesssim}
\def\gs{\gtrsim}

\def\hra{\hookrightarrow}

\def\M{\mathcal M}

\def\a{\alpha}
\def\b{\beta}
\def\o{\omega}

\def\la{\lambda}

\def\vp{\varphi}

\def\i{\infty}
\def\I{(0,\i)}
\def\t{\theta}

\def\N{\mathbb N}

\def\R{\mathbb R}

\def\R{\mathbb R}

\def\M{\mathfrak M}

\def\W{W}

\def\mp{{\mathfrak M}}

\def\rn{\R^n}
\def\a{\alpha}

\def\b{\beta}
\def\o{\omega}
\def\O{\Omega}

\def\la{\lambda}

\def\vp{\varphi}

\def\i{\infty}
\def\I{(0,\i)}
\def\t{\theta}

\def\dual{\,^{^{\mathsf{c}}}\!}

\def\Btd {\dual\Bt}
\def\Brd {\dual\Br}

\def\Bt {{B(0,t)}}

\def\Bxt {{B(x,t)}}
\def\Br {{B(0,r)}}
\def\Bxr {{B(x,r)}}

\def\LM {LM_{p\t,\o}}
\def\LMb {LM_{p\t_1,\o_1}(\rn)}
\def\LMi {LM_{p\t_2,\o_2}(\rn)}
\def\LMd {\dual{\LM}}

\def\LM {LM_{p\t,\o}}
\def\LMv {LM_{p\t,\o}(\rn,v)}
\def\LMd {\dual{\LM}}
\def\LMvd {\dual{\LMv}}
\def\Ot {\Omega_{\t}}
\def\Otd {\dual{\Ot}}

\begin{document}

\title[]
{Embeddings between weighted local Morrey-type spaces and weighted Lebesgue spaces}

\author{R.Ch.~Mustafayev, T.~{\"U}nver}

\thanks{We would like to thank Professor V.I.~Burenkov for making available the reference \cite{BurGol} in preprint form.}

\begin{abstract}
In this paper, the embeddings between weighted local Morrey-type
spaces and weighted Lebesgue spaces are investigated.
\end{abstract}



\maketitle

\section{Introduction}\label{in}

Throughout the paper, we always denote by $c$ and $C$ a positive
constant, which is independent of main parameters but it may vary
from line to line. However a constant with subscript or
superscript such as $c_1$ does not change in different
occurrences. By $a\lesssim b$, ($b\gtrsim a$) we mean that $a\leq
\la b$, where $\la >0$ depends on inessential parameters. If
$a\lesssim b$ and $b\lesssim a$, we write $a\approx b$ and say
that $a$ and $b$ are equivalent.  We will denote by $\bf 1$ the
function ${\bf 1}(x) = 1$, $x \in \rn$. For $x\in\rn$ and $r>0$,
let $\Bxr$ be the open ball centered at $x$ of radius $r$ and
$\dual{\Bxr}:=\rn\backslash \Bxr$.

Let $A,\,B$ be some sets and $\vp,\,\psi$ be non-negative
functions defined on $A \times B$ (It may happen that $\vp (\a,\b)
= \i$ or $\psi (\a,\b) = \i$ for some $\a \in A$, $\b \in B$). We
say that $\vp$ is dominated by $\psi$ (or $\psi$ dominates $\vp$)
on $A \times B$ uniformly in $\a \in A$ and write
$$
\vp (\a,\b) \ls \psi (\a,\b) \qq \mbox{uniformly in} \qq \a \in A
$$
or
$$
\psi (\a,\b) \gs \vp (\a,\b) \qq \mbox{uniformly in} \qq \a \in A,
$$
if for each $\b \in B$ there exists $C(\b) > 0$ such that
$$
\vp (\a,\b) \le C(\b) \psi (\a,\b)
$$
for all $\a \in A$. We also say that $\vp$ is equivalent to $\psi$
on $A \times B$ uniformly in $\a \in A$ and write
$$
\vp (\a,\b) \ap \psi (\a,\b) \qq \mbox{uniformly in} \qq \a \in A,
$$
if $\vp$ and $\psi$ dominate each other on $A \times B$ uniformly
in $\a \in A$ (see, for instance, \cite{BurGol}).

Given two quasi-normed vector spaces $X$ and $Y$, we write $X=Y$
if $X$ and $Y$ are equal in the algebraic and the topological
sense (their quasi-norms are equivalent). The symbol
$X\hookrightarrow Y$ ($Y \hookleftarrow X$) means that $X\subset
Y$ and the natural embedding $\Id$ of $X$ in $Y$ is continuous,
that is, there exist a constant $c > 0$ such that $\|z\|_Y \le
c\|z\|_X$ for all $z\in X$. The best constant of the embedding
$X\hookrightarrow Y$ is $\|\Id\|_{X \rw Y}$.

Let $A$ be any measurable subset of $\rn$, $n \ge 1$. By $\mp (A)$
we denote the set of all measurable functions on $A$. The symbol
$\mp^+ (A)$ stands for the collection of all $f\in\mp (A)$ which
are non-negative on $A$. The family of all weight functions (also
called just weights) on $A$, that is, measurable, positive and
finite a.e. on $A$, is given by $\W (A)$.

For $p\in (0,\i]$ and $w\in \mp^+(A)$, we define the functional
$\|\cdot\|_{p,A,w}$ on $\mp (A)$ by
\begin{equation*}
\|f\|_{p,A,w} : = \bigg\{\begin{array}{cl}
                             \left(\int_A |f(x)|^p w(x)\,dx \right)^{1/p} & \qq\mbox{if}\qq p<\i \\
                             \esup_{A} |f(x)|w(x) & \qq\mbox{if}\qq p=\i
                           \end{array}
\bigg..
\end{equation*}

If, in addition, $w\in \W(A)$, then the weighted Lebesgue space
$L_p(A,w)$ is given by
\begin{equation*}
L_{p,w}(A) \equiv L_p(A,w) : = \{f\in \mp (A):\,\, \|f\|_{p,A,w} <
\i\}
\end{equation*}
and it is equipped with the quasi-norm $\|\cdot\|_{p,A,w}$.

When $w\equiv 1$ on $A$, we write simply $L_p(A)$ and
$\|\cdot\|_{p,A}$ instead of $L_p(A,w)$ and $\|\cdot\|_{p,A,w}$,
respectively.

We adopt the following usual conventions.
\begin{conv}\label{Notat.and.prelim.conv.1.1}
{\rm (i)} Throughout the paper we put $0/0 = 0$, $0 \cdot (\pm \i)
= 0$ and $1 / (\pm\i) =0$.

{\rm (ii)} We put
$$
p' : = \left\{\begin{array}{cl} \frac p{1-p} & \text{if} \quad 0<p<1,\\
                      +\infty &\text{if}\quad p=1, \\
                   \frac p{p-1}  &\text{if}\quad 1<p<+\infty,\\
          1  &\text{if}\quad p=+\infty.
\end{array}
\right.
$$

{\rm (iii)} If $I = (a,b) \subseteq \R$ and $g$ is a monotone
function on $I$, then by $g(a)$ and $g(b)$ we mean the limits
$\lim_{x\rw a+}g(x)$ and $\lim_{x\rw b-}g(x)$, respectively.
\end{conv}

Morrey-type spaces, appeared to be quite useful in the study of
the local behavior of the solutions to partial differential
equations, a priori estimates and other topics in the theory of
PDE, were widely investigated during last decades. On the one
hand, the research includes the study of classical operators of
Harmonic Analysis - maximal, singular and potential operators - in
these spaces (see, for instance, \cite{Gul} - \cite{GulMus1},
\cite{BurHus0} - \cite{BurJainTar}, \cite{BurGol} and
\cite{Bur1}), on the other hand, the functional-analytic
properties of Morrey-type spaces and relation of these spaces with
other known function spaces are studied (see, for instance,
\cite{BurNur}, \cite{GogMus1}, \cite{GogMus3}).

Let us recall definitions of weighted local Morrey-type spaces and
weighted complementary local Morrey-type spaces.
\begin{defi}\label{defi.2.1}
Let  $0 <p, \theta \le \infty$, $\o \in \mp^+ \I$ and $v \in
\W(\rn)$. We denote by $\LMv$ the weighted local Morrey-type
space, the set of all $f\in L_{p,v}^{\loc}(\rn)$ with
$$
\left\| f\right\|_{\LMv} : = \left\| \o(r) \|f\|_{p,v,\Br}
\right\|_{\t,(0,\infty)} < \i.
$$
\end{defi}

\begin{defi}\label{defi.2.2}
Let  $0 <p, \theta \le \infty$, $\o \in \mp^+ \I$ and $v \in
\W(\rn)$. We denote by $\LMvd$ the weighted complementary local
Morrey-type space, the set of all functions such that $f\in
L_{p,v}(\Btd)$ for all $t>0$ with
$$
\left\| f\right\|_{\LMvd} : = \left\| \o(r) \|f\|_{p,v,\Brd}
\right\|_{\t,(0,\infty)} < \i.
$$
\end{defi}
\begin{rem}
In \cite{BurHus1} and \cite{BurGulHus1} it were proved that the
spaces $\LM (\rn) : = \LM (\rn, \bf 1)$ and $\LMd (\rn) : = \LMd
(\rn, \bf 1)$ are non-trivial, i.e. consists not only of functions
equivalent to $0$ on $\rn$, if and only if
\begin{equation}\label{11111}
\|\o\|_{\t,(t,\i)} < \i, \qq \mbox{for some} \qq t >0,
\end{equation}
and
\begin{equation}\label{1111111}
\|\o\|_{\t,(0,t)} < \i,\qq \mbox{for some} \qq t >0,
\end{equation}
respectively. The same conclusion can be drawn for $\LMv$ and
$\LMvd$ for any $v \in \W(\rn)$.
\end{rem}
\begin{defi}
Let $0<p,\t\leq \i$. We denote by $\Ot$ the set all functions $\o
\in \mp^+ \I$ such that
$$
0<\|\o\|_{\t,(t,\i)}<\i,~~ t>0,
$$
and by $\Otd$ the set all functions $\o \in \mp^+ \I$ such that
$$
0<\|\o\|_{\t,(0,t)}<\i,~~ t>0.
$$
\end{defi}
Let $v \in \W(\rn)$. It is easy to see that $\LMv$ and $\LMvd$ are
quasi-normed vector spaces when $\o \in \Ot$ and $\o \in \Otd$,
respectively.

We recall that  $\LMv$ and $\LMvd$ coincide with some weighted
Lebesgue spaces.
\begin{thm}\label{thm.2.7.}
Let $1\le p < +\infty$, $\o \in \O_{p}$, $v \in \W(\rn)$. Then
$$
LM_{pp,\o}(\rn,v) = L_{p}(\rn,u),
$$
and norms are equivalent, where
$$
u(x) = v(x) \|\o\|_{p,(|x|,\i)}^p.
$$
\end{thm}

\begin{thm}\label{thm.2.7.00}
Let $1\le p < +\infty$, $\o \in \dual{\O}_{p}$, $v \in \W(\rn)$.
Then
$$
\dual{LM}_{pp,\o}(\rn,v) = L_{p}(\rn,u),
$$
and norms are equivalent, where
$$
u(x) = v(x) \|\o\|_{p,(0,|x|)}^p.
$$
\end{thm}
Note that Theorems \ref{thm.2.7.} and \ref{thm.2.7.00} were proved
in \cite{GulMus1}
when $v = \bf 1$.

Let $f \in L_1^{\loc}(\rn)$. The maximal operator $M$ is defined
for all $x \in \rn$ by
$$
Mf(x) : = \sup_{t > 0} \frac{1}{|\Bxt|} \int_{\Bxt} |f(y)|\,dy,
$$
where $|\Bxt|$ is the Lebesgue measure of the ball $\Bxt$.

The boundedness of the maximal operator from $\LMb$ to $\LMi$ for
general $\o_1$ and $\o_2$ was studied in \cite{BurHus0},
\cite{BurHus1}, \cite{BGGM} and \cite{BurGol}. In \cite{BurHus0},
\cite{BurHus1}, \cite{BGGM}, for a certain range of the parameters
$p,\,\t_1$ and $\t_2$, necessary and sufficient conditions on
$\o_1$ and $\o_2$ were obtained ensuring the boundedness of $M$
from $\LMb$ to $\LMi$, namely the following statement was proved.
\begin{thm}\label{BurGul.Theorem}
If $n \in \N$, $1 < p < \i$, $0 < \t_1 \le \t_2 \le \i$, $\o_1 \in
\O_{\t_1}$, and $\o_2 \in \O_{\t_2}$, then the condition
\begin{equation}\label{BurGul.condition}
\bigg\| \o_2(r) \bigg( \frac{r}{t+r}\bigg)^{n/p} \bigg\|_{\t_2,\I}
\ls \|\o_1\|_{\t_1,(t,\i)}
\end{equation}
uniformly in $t \in \I$ is necessary and sufficient for the
boundedness of $M$ from $\LMb$ to $\LMi$. Moreover,
\begin{equation*}
\|M\|_{\LMb \rw\LMi} \ap \sup_{t \in \I} \|\o_1\|_{\t_1,(t,\i)}^{-1}
\bigg\| \o_2(r) \bigg( \frac{r}{t+r}\bigg)^{n/p} \bigg\|_{\t_2,\I}
\end{equation*}
uniformly in $\o_1 \in \O_{\t_1}$ and $\o_2 \in \O_{\t_2}$.
\end{thm}
In \cite{BurHus0}, \cite{BurHus1} this was proved under the
additional assumption $\t_1 \le p$. The general case was
considered in \cite{BGGM}.

If $\t_2 < \t_1$, then sufficient conditions on $\o_1$ and $\o_2$
for the boundedness of $M$ from $\LMb$ to $\LMi$ are given in
\cite{BGGM}. However, the problem of finding necessary and
sufficient condition on $\o_1$ and $\o_2$ ensuring the boundedness
of $M$ from $\LMb$ to $\LMi$ for the case $\t_2 < \t_1$ is still
open. In \cite{BurGol} the solution of this problem is given for
very particular case in which $\t_1 = \i$ and $\o_1 (r) \equiv 1$.
In other words, for all admissible values of the parameters $p_1$,
$p_2$ and $\t$ authors find necessary and sufficient conditions on
$\o$ ensuring the boundedness of the  maximal operator from
$L_{p_1}(\rn) = LM_{p_1 \i, \bf 1}(\rn)$ to $LM_{p_2 \t,
\o}(\rn)$.
\begin{thm}[\cite{BurGol}, see
also \cite{Bur1}]\label{BurGol}
Let $n \in \N$, $0 < p_2 \le p_1
\le \i$, $0 < \t \le \i$, and $\o \in \O_{\t}$.

{\rm 1.} If $1 < p_2 = p_1$, $0 < \t \le \i$ or $0 < p_2 < p_1$,
$1 < p_1$, $\t = \i$, then
\begin{equation}\label{BurGol.eq.1}
\|M\|_{L_{p_1}(\rn) \rw LM_{p_2 \t, \o}(\rn)} \ap
\left\|r^{n(1/{p_2} - 1 / {p_1})}\o(r)\right\|_{\t,\I}
\end{equation}
uniformly in $\o \in \O_{\t}$.

In particular, if $1 < p \le \i$, $0 < \t \le \i$. then
$$
\|M\|_{L_{p}(\rn) \rw LM_{p \t, \o}(\rn)} \ap
\left\|\o(r)\right\|_{\t,\I}
$$
uniformly in $\o \in \O_{\t}$.

{\rm 2.} If $0 < p_2 < p_1$, $1 < p_1$ and $\t < \i$, then
\begin{align}
\|M\|_{L_{p_1}(\rn) \rw LM_{p_2 \t, \o}(\rn)} & \ap
\left\|t^{n(1/{p_2} - 1 / {p_1}) -
1/s}\|\o\|_{\t,(t,\i)}\right\|_{s,\I}
\label{BurGol.eq.2}\\
& \ap \bigg\|t^{n(1/{p_2} - 1 / {p_1}) -
1/s}\bigg\|\bigg(\frac{r}{r+t}\bigg)^{n /
{p_2}}\o(r)\bigg\|_{\t,\I}\bigg\|_{s,\I} \notag\label{BurGol.eq.3}
\end{align}
uniformly in $\o \in \O_{\t}$, where
\begin{equation}\label{BurGol.eq.400}
s = \bigg\{ \begin{array}{cc}
  \frac{p_1 \t}{p_1 - \t} & \qq \t < p_1 \\
  \i & \qq \t \ge p_1 \\
\end{array}\bigg..
\end{equation}
\end{thm}
The idea used in \cite{BurGol} is mainly based on the following
theorems.
\begin{thm}[\cite{BurGol}, see
also \cite{Bur1}]\label{BurGol00}
Let $n \in \N$, $0 < p_2 \le p_1
\le \i$, $0 < \t \le \i$, and $\o \in \O_{\t}$.

{\rm 1.} If $p_2 = p_1$, $0 < \t \le \i$ or $0 < p_2 < p_1$, $\t =
\i$, then
\begin{equation}\label{BurGol.eq.100}
\|\Id\|_{L_{p_1}(\rn) \rw LM_{p_2 \t, \o}(\rn)} \ap
\left\|r^{n(1/{p_2} - 1 / {p_1})}\o(r)\right\|_{\t,\I}
\end{equation}
uniformly in $\o \in \O_{\t}$.

{\rm 2.} If $0 < p_2 < p_1$ and $\t < \i$, then
\begin{align}
\|\Id\|_{L_{p_1}(\rn) \rw LM_{p_2 \t, \o}(\rn)} & \ap
\left\|t^{n(1/{p_2} - 1 / {p_1}) -
1/s}\|\o\|_{\t,(t,\i)}\right\|_{s,\I} \label{BurGol.eq.200}
\end{align}
uniformly in $\o \in \O_{\t}$, where $s$ is defined by
\eqref{BurGol.eq.400}.
\end{thm}

\begin{thm}\label{reduction.lemma}
Let $X,\,Y$ be a quasi-normed vector spaces of measurable functions
on $\rn$ and let $M$ is bounded on $X$. Moreover, assume that $Y$
satisfies the monotonicity property, that is,
$$
0 \le g \le f \qq \Rightarrow \qq \|g\|_Y \ls \|f\|_Y.
$$
Then $M$ is bounded from $X$ to $Y$ if and only if
$X\hookrightarrow Y$, and
$$
\|M\|_{X \rw Y} \ap \|I\|_{X \rw Y}.
$$
\end{thm}
\begin{proof}
Since $|f| \le Mf$, by the lattice property of $Y$, we have that
$$
\|I\|_{X \rw Y} = \sup_{f\not\sim 0} \frac{\|f\|_Y}{\|f\|_X} \ls
\sup_{f\not\sim 0} \frac{\|Mf\|_Y}{\|f\|_X} = \|M\|_{X \rw Y}.
$$
On the other hand,
$$
\|M\|_{X \rw Y} = \sup_{f\not\sim 0} \frac{\|Mf\|_Y}{\|f\|_X} \le
\left(\sup_{f\not\sim 0} \frac{\|Mf\|_X}{\|f\|_X} \right)\|I\|_{X
\rw Y} = \|M\|_{X \rw X}\,\|I\|_{X \rw Y}.
$$
We have used that $\|g\|_Y \le \|g\|_X \|I\|_{X \rw Y}$ for any $g
\in X$.
\end{proof}
Note that in \cite{BurGol} Theorem \ref{reduction.lemma}
was proved for $X = L_{p_1}(\rn)$ and $Y = LM_{p_2 \t, \o}(\rn)$.

The aim of this paper is to characterize the embeddings between
weighted local Morrey-type spaces and weighted Lebesgue spaces, that
is, the embeddings
\begin{align}
L_{p_1}(\rn,v_1) & \hra LM_{p_2 \t, \o}(\rn,v_2), \label{emb1}\\
L_{p_1}(\rn,v_1) & \hra \dual{LM}_{p_2 \t, \o}(\rn,v_2), \label{emb2}\\
L_{p_1}(\rn,v_1) & \hookleftarrow LM_{p_2\t, \o}(\rn,v_2), \label{emb3}\\
L_{p_1}(\rn,v_1) & \hookleftarrow \dual{LM}_{p_2 \t, \o}(\rn,v_2).
\label{emb4}
\end{align}
The method of investigation is based on using the characterizations
of the direct and reverse multidimensional Hardy inequalities.

Our main results are Theorems \ref{main1}, \ref{main100},
\ref{main2} and \ref{main200}. Note that Theorem \ref{main1} is a
generalization of Theorem \ref{BurGol00} to the weighted case.
Theorems \ref{main100}, \ref{main2} and \ref{main200} are
characterizations of embeddings \eqref{emb2}, \eqref{emb3} and
\eqref{emb4}, respectively. Using Theorems \ref{main1} and
\ref{main100} we are able to calculate the norms
$\|M\|_{L_{p_1}(\rn,v_1) \rw LM_{p_2 \,\t, \o}(\rn,v_2)}$ and
$\|M\|_{L_{p_1}(\rn,v_1) \rw \dual{LM}_{p_2 \,\t, \o}(\rn,v_2)}$,
when $v_1$ is a weight function from the Muckenhoupt class
$A_{p_1}$, $1 < p_1 < \i$ (see Corollaries \ref{Corollary.1} and
\ref{Corollary.3.6}). Theorems \ref{main2} and \ref{main200} make
it possible to find the associate spaces of $\LMv$ and $\LMvd$
(see Theorems \ref{thm.6.6} and \ref{thm.6.5}).

The paper is organized as follows. Section \ref{sec2} contains
some preliminaries along with the standard ingredients used in the
proofs. In Sections \ref{sec3} and \ref{sec4} we give the
characterizations of the embeddings \eqref{emb1}, \eqref{emb2} and
\eqref{emb3}, \eqref{emb4},  respectively.


\section{Some Hardy-type inequalities}\label{sec2}

In \cite{ChristGraf}, M.~Christ and L.~Grafakos showed that the
$n$-dimensional Hardy inequality
$$
\int_{\rn} \left(\frac{1}{|B(0,|x|)} \int_{B(0,|x|)}
f(y)\,dy\right)^p\,dx \le \left( \frac{p}{p-1}\right)^p \int_{\rn}
f(x)^p\,dx, \qq 1 < p < \i,
$$
holds for all $f \in \mp^+(\rn)$, the constant $(
\frac{p}{p-1})^p$ being again the best possible. In
\cite{DrabHeinKuf}, P.~Dr\'{a}bek, H.P.~Heining and A.~Kufner
extended this Hardy inequality to general $n$-dimensional weights
$u,\,w$ and to the whole range of the parameters $p,\,q$, $1 < p <
\i$, $0 < q < \i$. The necessary and sufficient conditions for the
validity of the inequality
\begin{equation}\label{H_n}
\left( \int_{\rn}  \left(\int_{B(0,|x|)} f(y)\,dy\right)^q
\,u(x)\,dx \right)^{1/q} \le C \left( \int_{\rn} f(x)^p \,w(x)\,dx
\right)^{1/p}
\end{equation}
for all $f \in \mp^+(\rn)$ are exactly the analogous of the
corresponding conditions for dimension one. It is easy to see that
if $u$ is a function on $\rn$ such that $u(x) = v(|x|) /
|x|^{n-1}$, $x\in\rn$, for some function $v \in \mp^+\I$, then
\eqref{H_n} is equivalent to the following inequality
\begin{equation}
\left( \int_0^{\i}  \left(\int_{\Bt} f(y)\,dy\right)^q \,v(t)\,dt
\right)^{1/q} \le C \left( \int_{\rn} f(x)^p \,w(x)\,dx
\right)^{1/p}.
\end{equation}
According to the above remark, we can formulate the following
theorems.
\begin{thm}\label{lem.2.8}
Let $1 \le p \le \i$, $0 < q \le \i$, $v \in \mp^+\I$ and $w \in
\mp^+(\rn)$. Denote by
$$
(Hf) (t) : = \int_{\Bt} f(x)\,dx,  \qq f \in \mp^+(\rn),\qq t \ge
0.
$$
Then the inequality
\begin{equation}\label{eq.2.2}
\left\|Hf\right\|_{q,v,\I}\leq c \left\|f\right\|_{p,w,\rn}
\end{equation}
holds for all $f \in \mp^+(\rn)$ if and only if $A(p,q) < \i$, and
the best constant in \eqref{eq.2.2}, that is,
\begin{equation*}
B(p,q) : = \sup_{f\in \mp^+(\rn)} \left\|Hf\right\|_{q,v,\I} /
\left\|f\right\|_{p,w,\rn}
\end{equation*}
satisfies $B(p,q) \ap A(p,q)$, where

{\rm (a)} for $1 < p \le q < \i$,
$$
A(p,q): = \sup_{t > 0} \left( \int_t^{\i} v(s)\,ds \right)^{1/q}
\left( \int_{\Bt} w(x)^{1-p'}\,dx \right)^{1/ {p'}} \,;
$$

{\rm (b)} for $1 < p  < \i$, $0<q<p$  and $1 /r = 1 / q -  1 / p$,
$$
A(p,q) : = \left(\int_0^{\i}\left(\int_t^{\i}
v(s)ds\right)^{{r}/{p}}\,v(t)\,
\left(\int_{\Bt}{w(x)}^{1-p'}\,dx\right)^{{r}/{p'}}\,dt\right)^{{1}/{r}}
\,;
$$

{\rm (c)} for $1<p < \i$, $q = \i$,
$$
A(p,q): = \sup_{t > 0}\left(\esup_{t < s < \i} v(s)\right)
\left(\int_{\Bt}{w(x)}^{1-p'}\,dx\right)^{1/{p'}} \,;
$$

{\rm (d)} for $p = q = \i$,
$$
A(p,q) : =\sup_{t>0}\left(\esup_{t < s < \i} v(s)\right)
\int_{\Bt} \frac{dx}{w(x)} \, ;
$$

{\rm (e)} for $p = \i$, $0 < q  < \i$,
$$
A(p,q) : = \left(\int_0^{\i}v(t)\,
\left(\int_{\Bt}\frac{dx}{w(x)}\right)^q \,dt\right)^{{1}/{q}} \,;
$$

{\rm (f)} for $p = 1$, $1 \le q < \i$,
$$
A(p,q) : = \sup_{t \in \I}\left(\int_t^{\i}
v(s)ds\right)^{{1}/{q}} \esup_{x \in \Bt}w(x)^{-1} \,;
$$

{\rm (g)} for $p = 1$, $0 < q < 1$,
$$
A(p,q) : = \left(\int_0^\i \left(\int_t^{\i} v(s)\,
ds\right)^{q'}\, v(t) \left(\esup_{x \in
\Bt}w(x)^{-1}\right)^{q'}\,dt \right)^{{1}/{q'}} \,;
$$

{\rm (h)} for $p = 1$, $q = \i$,
$$
A(p,q) : = \sup_{t \in \I} \left(\esup_{t < s < \i}v(s)\right)
\left(\esup_{x \in \Bt}w(x)^{-1}\right).
$$
\end{thm}

\begin{thm}\label{lem.2.800}
Let $1 \le p \le \i$, $0 < q \le \i$, $v \in \mp^+\I$ and $w \in
\mp^+(\rn)$. Denote by
$$
(H^*f) (t) : = \int_{\Btd} f(x)\,dx,  \qq f \in \mp^+(\rn),\qq t
\ge 0.
$$
Then the inequality
\begin{equation}\label{eq.2.200}
\left\|H^*f\right\|_{q,v,\I}\leq c \left\|f\right\|_{p,w,\rn}
\end{equation}
holds for all $f \in \mp^+(\rn)$ if and only if $A^*(p,q) < \i$,
and the best constant in \eqref{eq.2.200}, that is,
\begin{equation*}
B^*(p,q) : = \sup_{f\in \mp^+(\rn)} \left\|H^*f\right\|_{q,v,\I} /
\left\|f\right\|_{p,w,\rn}
\end{equation*}
satisfies $B^*(p,q) \ap A^*(p,q)$. Here

{\rm (a)} for $1 < p \le q < \i$,
$$
A^*(p,q): = \sup_{t > 0} \left( \int_0^t v(s)\,ds \right)^{1/q}
\left( \int_{\Btd} w(x)^{1-p'}\,dx \right)^{1/ {p'}} \,;
$$

{\rm (b)} for $1 < p  < \i$, $0<q<p$  and $1 /r = 1 / q -  1 / p$,
$$
A^*(p,q) : = \left(\int_0^{\i}\left(\int_0^t
v(s)ds\right)^{{r}/{p}}\,v(t)\,
\left(\int_{\Btd}{w(x)}^{1-p'}\,dx\right)^{{r}/{p'}}\,dt\right)^{{1}/{r}}
\,;
$$

{\rm (c)} for $1<p < \i$, $q = \i$,
$$
A^*(p,q): = \sup_{t > 0}\left(\esup_{0 < s < t} v(s)\right)
\left(\int_{\Btd}{w(x)}^{1-p'}\,dx\right)^{1/{p'}} \,;
$$

{\rm (d)} for $p = q = \i$,
$$
A^*(p,q) : =\sup_{t>0}\left(\esup_{0 < s < t} v(s)\right)
\int_{\Btd} \frac{dx}{w(x)} \,;
$$

{\rm (e)} for $p = \i$, $0 < q  < \i$,
$$
A^*(p,q) : = \left(\int_0^{\i}v(t)\,
\left(\int_{\Btd}\frac{dx}{w(x)}\right)^q \,dt\right)^{{1}/{q}}
\,;
$$

{\rm (f)} for $p = 1$, $1 \le q < \i$,
$$
A^*(p,q) : = \sup_{t \in \I}\left(\int_0^t v(s)ds\right)^{{1}/{q}}
\esup_{x \in \Btd}w(x)^{-1} \,;
$$

{\rm (g)} for $p = 1$, $0 < q < 1$,
$$
A^*(p,q) : = \left(\int_0^\i \left(\int_0^t v(s)\,
ds\right)^{q'}\, v(t) \left(\esup_{x \in
\Btd}w(x)^{-1}\right)^{q'}\,dt \right)^{{1}/{q'}} \,;
$$

{\rm (h)} for $p = 1$, $q = \i$,
$$
A^*(p,q) : = \sup_{t \in \I} \left(\esup_{0 < s < t} v(s)\right)
\left(\esup_{x \in \Btd}w(x)^{-1}\right).
$$
\end{thm}

\begin{thm}\label{lem.2.8000000}
Let $0 < q \le \i$, $v \in \mp^+\I$ and $w \in \mp^+(\rn)$. Denote
by
$$
(Sf) (t) : = \esup_{x \in \Bt} f(x),  \qq f \in \mp^+(\rn),\qq t
\ge 0.
$$
Then the inequality
\begin{equation*}
\left\|(Sf)v\right\|_{q,\I}\leq c \left\|fw\right\|_{\i,\rn}
\end{equation*}
holds for all $f \in \mp^+(\rn)$ if and only if
$$
\bigg\| v(r)\bigg(\esup_{x \in \Br}w(x)^{-1}\bigg)\bigg\|_{q,\I} <
\i,
$$
and
\begin{equation*}
\sup_{f\in \mp^+(\rn)} \left\|(Sf)v\right\|_{q,\I} /
\left\|fw\right\|_{\i,\rn} \ap \bigg\| v(r)\bigg(\esup_{x \in
\Br}w(x)^{-1}\bigg)\bigg\|_{q,\I}.
\end{equation*}
\end{thm}

\begin{thm}\label{lem.2.800000000}
Let $0 < q \le \i$, $v \in \mp^+\I$ and $w \in \mp^+(\rn)$. Denote
by
$$
(S^*f) (t) : = \esup_{x \in \Btd} f(x),  \qq f \in \mp^+(\rn),\qq
t \ge 0.
$$
Then the inequality
\begin{equation*}
\left\|(S^*f)v\right\|_{q,\I}\leq c \left\|fw\right\|_{\i,\rn}
\end{equation*}
holds for all $f \in \mp^+(\rn)$ if and only if
$$
\bigg\| v(r)\bigg(\esup_{x \in \Brd}w(x)^{-1}\bigg)\bigg\|_{q,\I}
< \i,
$$
and
\begin{equation*}
\sup_{f\in \mp^+(\rn)} \left\|(S^*f)v\right\|_{q,\I} /
\left\|fw\right\|_{\i,\rn} \ap \bigg\| v(r)\bigg(\esup_{x \in
\Brd}w(x)^{-1}\bigg)\bigg\|_{q,\I}.
\end{equation*}
\end{thm}

For the convenience of the reader we repeat the relevant material
from \cite{GogMus2} without proofs, thus making our exposition
self-contained.


Let $\vp$ be non-decreasing and finite function on the interval $I
: = (a,b)\subseteq \R$. We assign to $\vp$ the function $\la$
defined on subintervals of $I$ by
\begin{align}
\la ([y,z]) & = \vp(z+) - \vp(y-), \notag\\
\la ([y,z)) & = \vp(z-) - \vp(y-), \label{Notat.and.prelim.eq.1.4}\\
\la ((y,z]) & = \vp(z+) - \vp(y+), \notag\\
\la ((y,z)) & = \vp(z-) - \vp(y+). \notag
\end{align}
$\la$ is a non-negative, additive and regular function of
intervals. Thus (cf. \cite{r}, Chapter 10), it admits a unique
extension to a non-negative Borel measure $\la$ on $I$.

Note also that the associated Borel measure can be determined,
e.g., only by putting
$$
\la ([y,z]) = \vp(z+) - \vp(y-) \qq \mbox{for any}\qq [y,z]\subset
I
$$
(since the Borel subsets of $I$ can be generated by subintervals
$[y,z]\subset I$).

If $J\subseteq I$, then the Lebesgue-Stieltjes integral $\int_J
f\,d\vp$ is defined as $\int_J f\,d\la$. We shall also use the
Lebesgue-Stieltjes integral $\int_J f\,d\vp$ when $\vp$ is a
non-increasing and finite on the interval $I$. In such a case we
put
$$
\int_J f\,d\vp : = - \int_J f\,d(-\vp).
$$

We adopt the following conventions.
\begin{conv} \label{conv:3.3}
Let $I=(a,b)\subseteq \R$, $f:I\to [0,\infty]$ and $h:I\to
[0,\i]$. Assume that $h$ is non-decreasing and left-continuous on
$I$. If $h:I\to [0,\infty)$, then the symbol $\int_{I}f\,dh$ means
the usual Lebesgue-Stieltjes integral (with the measure $\la$
associated to $h$ is given by $\la([\a,\b))= h(\b)-h(\a)$ if $[\a,
\b)\subset (a,b)$ -- cf. \eqref{Notat.and.prelim.eq.1.4}).
However, if $h = \infty$ on some subinterval $(c,b)$  with $c\in
I$, then we define $\int_{I}f\,dh$ only if $f=0$ on $[c,b)$ and we
put
$$\int_{I}f\,dh=\int_{(a,c)}f\,dh.$$
\end{conv}
\begin{conv}
Let $I=(a,b)\subseteq \R$, $f:I\to [0,+\infty]$ and $h:I\to
[-\infty,0]$. Assume that $h$ is non-decreasing and
right-continuous on $I$. If $h:I\to (-\infty,0]$, then the symbol
$\int_{I}f\,dh$ means the usual Lebesgue-Stieltjes integral.
However, if $h= -\infty$ on some subinterval
 $(a,c)$  with $c\in I$, then we define  $\int_{I}f\,dh$ only if $f=0$ on
$(a,c]$ and we put
$$\int_{I}f\,dh=\int_{(c,b)}f\,dh.$$
\end{conv}

\begin{thm}\label{T:5.1}
Let  $w \in \mp^+(\rn)$ and $u \in \mp^+\I$ be such that
$\|u\|_{q,(t,\i)} <\infty$ for all $t\in (0,\i)$.

{\rm (a)} Assume that $0< q\le p\le1$. Then
\begin{equation} \label{5.1}
\Vert gw \Vert _{p,\rn}\leq c \left\Vert(Hg)u\right\Vert _{q,\I}
\end{equation}
 holds for all $g \in \mp^+(\rn)$ if and only if
\begin{equation} \label{5.2}
C(p,q) : =\sup _{t\in (0,\i)}\| w \|_{{p'},\Btd} \|
u\|_{q,(t,\i)}^{-1} <\infty .
\end{equation}
The best possible constant in \eqref{5.1}, that is,
\begin{equation*}
D(p,q) : = \sup _{g \in \mp^+(\rn)}\| gw \|_{p,\rn} / \|
(Hg)u\|_{q,\I}
\end{equation*}
satisfies $D(p,q)\approx C(p,q)$.

{\rm (b)} Let $0< p\le1$, $p<q\le\infty$ and $1/r = 1 / p - 1 /
q$. Then \eqref{5.1}  holds if and only if
$$
C(p,q) : =\left( \int_{\I} \| w \|_{{p'},\Btd}^{r}
\,d\left(\|u\|_{q,(t-,\i) }^{-r}\right) \right) ^{1 / {r}}
+\frac{\|w \|_{{p'},\rn}} {\|u\|_{q,\I}}<\infty,
$$
and $D(p,q)\approx C(p,q)$, where
\begin{equation*}
\|u\|_{q,(t-,\i)}:=\lim_{s\rw t-}\|u\|_{q,(s,\i)}, \qquad t\in \I.
\end{equation*}
\end{thm}

\begin{thm}\label{T:5.100}
Let  $w \in \mp^+(\rn)$ and $u \in \mp^+\I$ be such that
$\|u\|_{q,(0,t)} <\infty$ for all $t\in (0,\i)$.

{\rm (a)} Assume that $0< q\le p\le1$. Then
\begin{equation} \label{5.100}
\Vert gw \Vert _{p,\rn}\leq c \left\Vert(H^*g)u\right\Vert
_{q,(0,\i)}
\end{equation}
 holds for all $g \in \mp^+(\rn)$ if and only if
\begin{equation} \label{5.200}
C^*(p,q) : =\sup _{t\in (0,\i)}\| w \|_{{p'},\Bt} \|
u\|_{q,(0,t)}^{-1} <\infty .
\end{equation}
The best possible constant in \eqref{5.100}, that is,
\begin{equation*}
D^*(p,q) : = \sup _{g \in \mp^+(\rn)}\| gw \|_{p,\rn} / \|
(H^*g)u\|_{q,\I}
\end{equation*}
satisfies $D^*(p,q)\approx C^*(p,q)$.

{\rm (b)} Let $0< p\le1$, $p<q\le\infty$ and $1/r = 1 / p - 1 /
q$. Then \eqref{5.100}  holds if and only if
$$
C^*(p,q) : =\left( \int_{\I} \| w \|_{{p'},\Bt)}^{r}
\,d\left(-\|u\|_{q,(0,t+) }^{-r}\right) \right) ^{1 / {r}}
+\frac{\|w \|_{{p'},\rn}} {\|u\|_{q,\I}}<\infty,
$$
and $D^*(p,q)\approx C^*(p,q)$, where
\begin{equation*}
\|u\|_{q,(0,t+)}:=\lim_{s\rw t+}\|u\|_{q,(0,s)}, \qquad t\in \I.
\end{equation*}
\end{thm}

\begin{rem} \label{R:5.5}
Let $q<\infty$ in Theorems~\ref{T:5.1} and \ref{T:5.100}. Then
$$
\| u\|_{q,(t-,\i)}=\| u\|_{q,(t,\i)}\quad \mbox{and} \qq
\|u\|_{q,(0,t+) } = \|u\|_{q,(0,t)} \qq \mbox{for all} \quad t\in
\I,
$$
which implies that
$$
C(p,q) = \left( \int_{(a,b)} \Vert w \Vert _{{p'},\Btd}^{r}
\,d\left(\Vert u\Vert _{q,(t,\i)}^{-r}\right) \right) ^{1 /r}
+\frac{\Vert w \Vert _{{p'},\rn}} {\Vert u\Vert_{q,\I}},
$$
and
$$
C^*(p,q) = \left( \int_{\I} \| w \|_{{p'},\Bt)}^{r}
\,d\left(-\|u\|_{q,(0,t) }^{-r}\right) \right) ^{1 / {r}}
+\frac{\|w \|_{{p'},\rn}} {\|u\|_{q,\I}}.
$$
\end{rem}

\section{Characterizations of $L_{p_1}(\rn,v_1) \hra LM_{p_2 \t,
\o}(\rn,v_2)$ and $L_{p_1}(\rn,v_1) \hra \dual{LM}_{p_2 \t,
\o}(\rn,v_2)$}\label{sec3}

In this section we characterize \eqref{emb1} and \eqref{emb2}.
\begin{thm}\label{main1}
Let $0 < p_1 \le \i$, $0 < p_2 < \i$, $0 < \t \le \i$, $v_1,\,v_2
\in \W(\rn)$ and $\o \in \O_{\t}$.

{\rm (i)} If $p_2 < p_1 \le \t < \i$, then
\begin{equation*}
\|\Id\|_{L_{p_1}(\rn,v_1) \rw LM_{p_2 \, \t, \o}(\rn,v_2)} \ap
\sup_{t
> 0} \|\o\|_{\t,(t,\i)}
\left\|v_1^{-1/{p_1}}v_2^{1/{p_2}}\right\|_{\frac{p_1 p_2}{p_1 -
p_2},\Bt}
\end{equation*}
uniformly in $\o \in \O_{\t}$.

{\rm (ii)} If $p_2 < p_1 < \i$, $0 < \t < p_1$, then
\begin{equation*}
\|\Id\|_{L_{p_1}(\rn,v_1) \rw LM_{p_2 \,\t, \o}(\rn,v_2)} \ap
\left\|\|\o\|_{\t,(t,\i)}^{{\t}/{p_1}}\left\|v_1^{-1/{p_1}}v_2^{1/{p_2}}\right\|_{\frac{p_1
p_2}{p_1 - p_2},\Bt}\right\|_{\frac{\t p_1}{p_1 - \t},\o^{\t},\I}
\end{equation*}
uniformly in $\o \in \O_{\t}$.

{\rm (iii)} If $p_2 < p_1 < \i$, $\t = \i$, then
\begin{equation*}
\|\Id\|_{L_{p_1}(\rn,v_1) \rw LM_{p_2 \,\t, \o}(\rn,v_2)} \ap
\sup_{t > 0} \|\o\|_{\t,(t,\i)}
\left\|v_1^{-1/{p_1}}v_2^{1/{p_2}}\right\|_{\frac{p_1 p_2}{p_1 -
p_2},\Bt}
\end{equation*}
uniformly in $\o \in \O_{\t}$.

{\rm (iv)} If $p_1 = \t = \i$, $0 < p_2 < \i$, then
\begin{equation*}
\|\Id\|_{L_{p_1}(\rn,v_1) \rw LM_{p_2 \,\t, \o}(\rn,v_2)} \ap
\sup_{t > 0} \|\o\|_{\t,(t,\i)}
\left\|v_1^{-1}v_2^{1/{p_2}}\right\|_{p_2,\Bt}
\end{equation*}
uniformly in $\o \in \O_{\t}$.

{\rm (v)} If $p_1 = \i$, $0 < p_2 < \i$, $0 < \t < \i$, then
\begin{equation*}
\|\Id\|_{L_{p_1}(\rn,v_1) \rw LM_{p_2 \,\t, \o}(\rn,v_2)} \ap
\left\| \o(t) \left\| v_1^{-1}v_2^{1/{p_2}}  \right\|_{p_2,\Bt}
\right\|_{\t,\I}
\end{equation*}
uniformly in $\o \in \O_{\t}$.

{\rm (vi)} If $p : = p_1 = p_2 \le \t < \i$, then
\begin{equation*}
\|\Id\|_{L_{p_1}(\rn,v_1) \rw LM_{p_2 \,\t, \o}(\rn,v_2)} \ap
\sup_{t > 0} \|\o\|_{\t,(t,\i)} \| v_1^{-1 / p} v_2^{ 1/ p}
\|_{\i,\Bt}
\end{equation*}
uniformly in $\o \in \O_{\t}$.

{\rm (vii)} If $0 < p : = p_1 = p_2 < \i$,  $0 < \t < p$, then
\begin{equation*}
\|\Id\|_{L_{p_1}(\rn,v_1) \rw LM_{p_2 \,\t, \o}(\rn,v_2)} \ap
\left\|\|\o\|_{\t,(t,\i)}^{\t/{p}}\left\|v_1^{-1/{p}}v_2^{1/{p}}\right\|_{\i,\Bt}\right\|_{\frac{\t
p}{p - \t},\o^{\t},\I}
\end{equation*}
uniformly in $\o \in \O_{\t}$.

{\rm (viii)} If $0 < p : = p_1 = p_2 < \i$, $\t = \i$, then
\begin{equation*}
\|\Id\|_{L_{p_1}(\rn,v_1) \rw LM_{p_2 \,\t, \o}(\rn,v_2)} \ap
\sup_{t > 0} \|\o\|_{\t,(t,\i)} \| v_1^{- 1 / p} v_2^{ 1 / p}
\|_{\i,\Bt}
\end{equation*}
uniformly in $\o \in \O_{\t}$.

{\rm (ix)} If $p : = p_1 = p_2 = \i$, $0 < \t < \i$, then
\begin{equation*}
\|\Id\|_{L_{p_1}(\rn,v_1) \rw LM_{p_2 \,\t, \o}(\rn,v_2)} \ap
\left\| \o(t) \left\| v_1^{-1}v_2 \right\|_{\i,\Bt}
\right\|_{\t,\I}
\end{equation*}
uniformly in $\o \in \O_{\t}$.

\end{thm}
\begin{proof}
{(i) - (viii).} Denote by
$$
q_1 : = {p_1} / {p_2}, \qq q_2 : = {\t} / {p_2}, \qq w_1 : = v_1
v_2^{-q_1}, \qq w_2 : = \o^{\t}.
$$
Since
\begin{align*}
\|\Id\|_{L_{p_1}(\rn,v_1) \rw LM_{p_2 \t, \o}(\rn,v_2)} & =
\sup_{f \not \sim \,0} \frac{\left\| \o(r) \|f\|_{p_2,v_2,\Br}
\right\|_{\t,(0,\infty)}}{\|f\|_{p_1,v_1,\rn}} \\
& = \left(\sup_{g \not \sim \,0} \frac{\left\| H (|g|)
\right\|_{q_2,w_2,\I} }{\|g\|_{q_1,w_1,\rn}}\right)^{1 / {p_2}},
\end{align*}
it remains to apply Theorem \ref{lem.2.8}.

{(ix).} Note that
\begin{align*}
\|\Id\|_{L_{\i}(\rn,v_1) \rw LM_{\i \t, \o}(\rn,v_2)} & = \sup_{f
\not \sim \,0} \frac{\left\| \o(r) \|f\|_{\i,v_2,\Br}
\right\|_{\t,(0,\infty)}}{\|f\|_{\i,v_1,\rn}} \\
& = \sup_{g \not \sim \,0} \frac{\left\| (S(|g|))\o
\right\|_{\t,\I} }{\|gw\|_{\i,\rn}},
\end{align*}
where $w = v_1 / v_2$. The statement follows by Theorem
\ref{lem.2.8000000}.
\end{proof}


\begin{thm}\label{main100}
Let $0 < p_1, \t \le \i$, $0 < p_2 < \i$, $v_1,\,v_2 \in \W(\rn)$
and $\o \in \Otd$.

{\rm (i)} If $p_2 < p_1 \le \t < \i$, then
\begin{equation*}
\|\Id\|_{L_{p_1}(\rn,v_1) \rw \dual{LM}_{p_2 \,\t, \o}(\rn,v_2)}
\ap \sup_{t > 0} \|\o\|_{\t,(0,t)}
\left\|v_1^{-1/{p_1}}v_2^{1/{p_2}}\right\|_{\frac{p_1 p_2}{p_1 -
p_2},\Btd}
\end{equation*}
uniformly in $\o \in \Otd$.

{\rm (ii)} If $p_2 < p_1 < \i$, $0 < \t < p_1$, then
\begin{equation*}
\|\Id\|_{L_{p_1}(\rn,v_1) \rw \dual{LM}_{p_2 \,\t, \o}(\rn,v_2)}
\ap
\left\|\|\o\|_{\t,(0,t)}^{{\t}/{p_1}}\left\|v_1^{-1/{p_1}}v_2^{1/{p_2}}\right\|_{\frac{p_1
p_2}{p_1 - p_2},\Btd}\right\|_{\frac{\t p_1}{p_1 - \t},\o^{\t},\I}
\end{equation*}
uniformly in $\o \in \Otd$.

{\rm (iii)} If $p_2 < p_1 < \i$, $\t = \i$, then
\begin{equation*}
\|\Id\|_{L_{p_1}(\rn,v_1) \rw \dual{LM}_{p_2 \,\t, \o}(\rn,v_2)}
\ap \sup_{t > 0} \|\o\|_{\t,(0,t)}
\left\|v_1^{-1/{p_1}}v_2^{1/{p_2}}\right\|_{\frac{p_1 p_2}{p_1 -
p_2},\Btd}
\end{equation*}
uniformly in $\o \in \Otd$.

{\rm (iv)} If $p_1 = \t = \i$, $0 < p_2 < \i$, then
\begin{equation*}
\|\Id\|_{L_{p_1}(\rn,v_1) \rw \dual{LM}_{p_2 \,\t, \o}(\rn,v_2)}
\ap \sup_{t > 0} \|\o\|_{\t,(0,t)}
\left\|v_1^{-1}v_2^{1/{p_2}}\right\|_{p_2,\Btd}
\end{equation*}
uniformly in $\o \in \Otd$.

{\rm (v)} If $p_1 = \i$, $0 < p_2 < \i$, $0 < \t < \i$, then
\begin{equation*}
\|\Id\|_{L_{p_1}(\rn,v_1) \rw \dual{LM}_{p_2 \,\t, \o}(\rn,v_2)}
\ap \left\| \o(t) \left\| v_1^{-1}v_2^{1/{p_2}}
\right\|_{p_2,\Btd} \right\|_{\t,\I}
\end{equation*}
uniformly in $\o \in \Otd$.

{\rm (vi)} If $p = p_1 = p_2 \le \t < \i$, then
\begin{equation*}
\|\Id\|_{L_{p_1}(\rn,v_1) \rw \dual{LM}_{p_2 \,\t, \o}(\rn,v_2)}
\ap \sup_{t > 0} \|\o\|_{\t,(0,t)} \bigg\| v_1^{-1 / p} v_2^{ 1/
p} \bigg\|_{\i,\Btd}
\end{equation*}
uniformly in $\o \in \Otd$.

{\rm (vii)} If $0 < p : = p_1 = p_2 < \i$,  $0 < \t < p$, then
\begin{equation*}
\|\Id\|_{L_{p_1}(\rn,v_1) \rw \dual{LM}_{p_2 \,\t, \o}(\rn,v_2)}
\ap
\left\|\|\o\|_{\t,(0,t)}^{\t/{p}}\left\|v_1^{-1/{p}}v_2^{1/{p}}\right\|_{\i,\Btd}\right\|_{\frac{\t
p}{p - \t},\o^{\t},\I}
\end{equation*}
uniformly in $\o \in \Otd$.

{\rm (viii)} If $0 < p : = p_1 = p_2 < \i$, $\t = \i$, then
\begin{equation*}
\|\Id\|_{L_{p_1}(\rn,v_1) \rw \dual{LM}_{p_2 \,\t, \o}(\rn,v_2)}
\ap \sup_{t > 0} \|\o\|_{\t,(0,t)} \bigg\| v_1^{- 1 / p} v_2^{ 1 /
p} \bigg\|_{\i,\Btd}
\end{equation*}
uniformly in $\o \in \Otd$.

{\rm (ix)} If $p : = p_1 = p_2 = \i$, $0 < \t < \i$, then
\begin{equation*}
\|\Id\|_{L_{p_1}(\rn,v_1) \rw \dual{LM}_{p_2 \,\t, \o}(\rn,v_2)}
\ap \left\| \o(t) \left\| v_1^{-1}v_2 \right\|_{\i,\Btd}
\right\|_{\t,\I}
\end{equation*}
uniformly in $\o \in \Otd$.
\end{thm}

Theorem \ref{reduction.lemma} reduces the problem of boundedness of
$M$ from $L_{p_1}(\rn,v_1)$ to $LM_{p\t_2,\o_2}(\rn,v_2)$ and from
$L_{p_1}(\rn,v_1)$ to $\dual{LM}_{p\t_2,\o_2}(\rn,v_2)$ to the
characterizations of \eqref{emb1} and \eqref{emb2}, respectively,
when we know the boundedness of $M$ on $L_{p_1}(\rn,v_1)$. The
latter happens exactly when $v_1 \in A_{p_1}$, $1 < p_1 < \i$.

Let $w$ be a weight function and $1< p< \i$. We say that $w\in A_p$
if there exists a constant $c_p>0$ such that, for every ball
$B\subset \rn$,
$$
\left(\int_B w(x)dx\right)\left(\int_B
w(x)^{1-p'}dx\right)^{p-1}\leq c_p|B|^p.
$$
It is well known that the Muckenhoupt classes characterize the
boundedness of  $M$ on weighted Lebesgue spaces. Namely, $M$ is
bounded on $L_{p}(\rn,w)$ if and only if $w\in A_p$, $1<p<\i$ (see,
for instance, \cite{Mucken}).

The following statements are consequences of combination of Theorems
\ref{main1} and \ref{main100} with Theorem \ref{reduction.lemma}.
\begin{cor}\label{Corollary.1}
Let $1 < p_1 < \i$, $0 < p_2 < \i$, $0 < \t \le \i$, $v_2 \in
\W(\rn)$ and $\o \in \O_{\t}$. Moreover, assume that $v_1 \in
A_{p_1}$.

{\rm (i)} If $p_2 < p_1 \le \t < \i$, then
\begin{equation*}
\|M\|_{L_{p_1}(\rn,v_1) \rw LM_{p_2 \,\t, \o}(\rn,v_2)} \ap
\sup_{t
> 0} \|\o\|_{\t,(t,\i)}
\left\|v_1^{-1/{p_1}}v_2^{1/{p_2}}\right\|_{\frac{p_1 p_2}{p_1 -
p_2},\Bt}
\end{equation*}
uniformly in $\o \in \O_{\t}$.

{\rm (ii)} If $p_2 < p_1 < \i$, $0 < \t < p_1$, then
\begin{equation*}
\|M\|_{L_{p_1}(\rn,v_1) \rw LM_{p_2 \,\t, \o}(\rn,v_2)} \ap
\left\|\|\o\|_{\t,(t,\i)}^{{\t}/{p_1}}\left\|v_1^{-1/{p_1}}v_2^{1/{p_2}}\right\|_{\frac{p_1
p_2}{p_1 - p_2},\Bt}\right\|_{\frac{\t p_1}{p_1 - \t},\o^{\t},\I}
\end{equation*}
uniformly in $\o \in \O_{\t}$.

{\rm (iii)} If $p_2 < p_1 < \i$, $\t = \i$, then
\begin{equation*}
\|M\|_{L_{p_1}(\rn,v_1) \rw LM_{p_2 \,\t, \o}(\rn,v_2)} \ap
\sup_{t
> 0} \|\o\|_{\t,(t,\i)}
\left\|v_1^{-1/{p_1}}v_2^{1/{p_2}}\right\|_{\frac{p_1 p_2}{p_1 -
p_2},\Bt}
\end{equation*}
uniformly in $\o \in \O_{\t}$.

%

{\rm (iv)} If $p : = p_1 = p_2 \le \t < \i$, then
\begin{equation*}
\|M\|_{L_{p_1}(\rn,v_1) \rw LM_{p_2 \,\t, \o}(\rn,v_2)} \ap
\sup_{t > 0} \|\o\|_{\t,(t,\i)} \| v_1^{-1 / p} v_2^{ 1/ p}
\|_{\i,\Bt}
\end{equation*}
uniformly in $\o \in \O_{\t}$.

{\rm (v)} If $ p : = p_1 = p_2$,  $0 < \t < p$, then
\begin{equation*}
\|M\|_{L_{p_1}(\rn,v_1) \rw LM_{p_2 \,\t, \o}(\rn,v_2)} \ap
\left\|\|\o\|_{\t,(t,\i)}^{\t/{p}}\left\|v_1^{-1/{p}}v_2^{1/{p}}\right\|_{\i,\Bt}\right\|_{\frac{\t
p}{p - \t},\o^{\t},\I}
\end{equation*}
uniformly in $\o \in \O_{\t}$.

{\rm (vi)} If $p : = p_1 = p_2$, $\t = \i$, then
\begin{equation*}
\|M\|_{L_{p_1}(\rn,v_1) \rw LM_{p_2 \,\t, \o}(\rn,v_2)} \ap
\sup_{t > 0} \|\o\|_{\t,(t,\i)} \| v_1^{- 1 /  p} v_2^{ 1 / p}
\|_{\i,\Bt}
\end{equation*}
uniformly in $\o \in \O_{\t}$.
\end{cor}
%
%
\begin{cor}\label{Corollary.3.6}
Let $1 < p_1 < \i$, $0 < p_2 < \i$, $0 < \t \le \i$, $v_2 \in
\W(\rn)$ and $\o \in \Otd$. Moreover, assume that $v_1 \in
A_{p_1}$.

{\rm (i)} If $p_2 < p_1 \le \t < \i$, then
\begin{equation*}
\|M\|_{L_{p_1}(\rn,v_1) \rw \dual{LM}_{p_2 \,\t, \o}(\rn,v_2)} \ap
\sup_{t > 0} \|\o\|_{\t,(0,t)}
\left\|v_1^{-1/{p_1}}v_2^{1/{p_2}}\right\|_{\frac{p_1 p_2}{p_1 -
p_2},\Btd}
\end{equation*}
uniformly in $\o \in \Otd$.

{\rm (ii)} If $p_2 < p_1 < \i$, $0 < \t < p_1$, then
\begin{equation*}
\|M\|_{L_{p_1}(\rn,v_1) \rw \dual{LM}_{p_2 \,\t, \o}(\rn,v_2)} \ap
\left\|\|\o\|_{\t,(0,t)}^{{\t} /
{p_1}}\left\|v_1^{-1/{p_1}}v_2^{1/{p_2}}\right\|_{\frac{p_1
p_2}{p_1 - p_2},\Btd}\right\|_{\frac{\t p_1}{p_1 - \t},\o^{\t},\I}
\end{equation*}
uniformly in $\o \in \Otd$.

{\rm (iii)} If $p_2 < p_1 < \i$, $\t = \i$, then
\begin{equation*}
\|M\|_{L_{p_1}(\rn,v_1) \rw \dual{LM}_{p_2 \,\t, \o}(\rn,v_2)} \ap
\sup_{t > 0} \|\o\|_{\t,(0,t)}
\left\|v_1^{-1/{p_1}}v_2^{1/{p_2}}\right\|_{\frac{p_1 p_2}{p_1 -
p_2},\Btd}
\end{equation*}
uniformly in $\o \in \Otd$.

%

{\rm (iv)} If $p : = p_1 = p_2 \le \t < \i$, then
\begin{equation*}
\|M\|_{L_{p_1}(\rn,v_1) \rw \dual{LM}_{p_2 \,\t, \o}(\rn,v_2)} \ap
\sup_{t > 0} \|\o\|_{\t,(0,t)} \bigg\| v_1^{-1 / p} v_2^{ 1/ p}
\bigg\|_{\i,\Btd}
\end{equation*}
uniformly in $\o \in \Otd$.

{\rm (v)} If $p : = p_1 = p_2$,  $0 < \t < p$, then
\begin{equation*}
\|M\|_{L_{p_1}(\rn,v_1) \rw \dual{LM}_{p_2 \,\t, \o}(\rn,v_2)} \ap
\left\|\|\o\|_{\t,(0,t)}^{\t/{p}}\left\|v_1^{-1/{p}}v_2^{1/{p}}\right\|_{\i,\Btd}\right\|_{\frac{\t
p}{p - \t},\o^{\t},\I}
\end{equation*}
uniformly in $\o \in \Otd$.

{\rm (vi)} If $p : = p_1 = p_2$, $\t = \i$, then
\begin{equation*}
\|M\|_{L_{p_1}(\rn,v_1) \rw \dual{LM}_{p_2 \,\t, \o}(\rn,v_2)} \ap
\bigg\| \|\o\|_{\t,(0,t)} \bigg\| v_1^{- 1 / p} v_2^{ 1 / p}
\bigg\|_{\i,\Btd} \bigg\|_{\i,\I}
\end{equation*}
uniformly in $\o \in \Otd$.
\end{cor}

\section{Characterizations of $LM_{p_2 \t, \o}(\rn,v_2) \hra
L_{p_1}(\rn,v_1)$ and $\dual{LM}_{p_2 \t, \o}(\rn,v_2) \hra
L_{p_1}(\rn,v_1)$}\label{sec4}

In this section we characterize the embeddings \eqref{emb3} and
\eqref{emb4}.
\begin{thm}\label{main2}
Let $0 < p_1 \le p_2 < \i$, $0 < \t \le \i$, $v_1,\,v_2 \in
\W(\rn)$ and $\o \in \O_{\t}$.

{\rm (a)} If $0 < \t \le p_1 \le p_2 < \i$, then
\begin{equation*}
\|\Id\|_{LM_{p_2 \,\t, \o}(\rn,v_2) \rw L_{p_1}(\rn,v_1)} \ap
\sup_{t > 0} \|\o\|_{\t,(t,\i)}^{-1}
\left\|v_1^{1/{p_1}}v_2^{-1/{p_2}}\right\|_{\frac{p_1 p_2}{p_2 -
p_1},\Btd}
\end{equation*}
uniformly in $\o \in \O_{\t}$.

{\rm (b)} If $0 < p_1 \le p_2 < \i$, $p_1 < \t \le \i$, then
\begin{align*}
\|\Id\|_{LM_{p_2 \,\t, \o}(\rn,v_2) \rw L_{p_1}(\rn,v_1)} \ap &
\left( \int_{\I}
\left\|v_1^{1/{p_1}}v_2^{-1/{p_2}}\right\|_{\frac{p_1 p_2}{p_2 -
p_1},\Btd}^{\frac{p_1 \t}{\t - p_1}} d \left(
\|\o\|_{\t,(t-,\i)}^{-\frac{p_1 \t}{\t - p_1}} \right)
\right)^{\frac{\t - p_1}{p_1 \t}} \\
& +  \|\o\|_{\t,(0,\i)}^{-1}
\left\|v_1^{1/{p_1}}v_2^{-1/{p_2}}\right\|_{\frac{p_1 p_2}{p_2 -
p_1},\rn}
\end{align*}
uniformly in $\o \in \O_{\t}$.
\end{thm}
\begin{proof}
Since
\begin{align*}
\|\Id\|_{LM_{p_2 \,\t, \o}(\rn,v_2) \rw L_{p_1}(\rn,v_1)} & =
\sup_{f \not \sim \,0} \frac{\|f\|_{p_1,v_1,\rn}}{\left\| w(r)
\|f\|_{p_2,v_2,\Br}
\right\|_{\t,(0,\infty)}} \\
& = \left(\sup_{g \not \sim \,0}
\frac{\|g\|_{q_1,w_1,\rn}}{\left\| H(|g|) \right\|_{q_2,w_2,\I}
}\right)^{1 / {p_2}},
\end{align*}
it remains to apply Theorem \ref{T:5.1}.
\end{proof}

\begin{thm}\label{main200}
Let $0 < p_1 \le p_2 < \i$, $0 < \t \le \i$, $v_1,\,v_2 \in
\W(\rn)$ and $\o \in \Otd$.

{\rm (a)} If $0 < \t \le p_1 \le p_2 < \i$, then
\begin{equation*}
\|\Id\|_{\dual{LM}_{p_2 \,\t, \o}(\rn,v_2) \rw L_{p_1}(\rn,v_1)}
\ap \sup_{t > 0} \|\o\|_{\t,(0,t)}^{-1}
\left\|v_1^{1/{p_1}}v_2^{-1/{p_2}}\right\|_{\frac{p_1 p_2}{p_2 -
p_1},\Bt}
\end{equation*}
uniformly in $\o \in \Otd$.

{\rm (b)} If $0 < p_1 \le p_2 < \i$, $p_1 < \t \le \i$, then
\begin{align*}
\|\Id\|_{\dual{LM}_{p_2 \,\t, \o}(\rn,v_2) \rw L_{p_1}(\rn,v_1)}
\ap & \left( \int_{\I}
\left\|v_1^{1/{p_1}}v_2^{-1/{p_2}}\right\|_{\frac{p_1 p_2}{p_2 -
p_1},\Bt}^{\frac{p_1 \t}{\t - p_1}} d \left( -
\|\o\|_{\t,(0,t+)}^{-\frac{p_1 \t}{\t - p_1}} \right)
\right)^{\frac{\t - p_1}{p_1 \t}} \\
& +  \|\o\|_{\t,(0,\i)}^{-1}
\left\|v_1^{1/{p_1}}v_2^{-1/{p_2}}\right\|_{\frac{p_1 p_2}{p_2 -
p_1},\rn}
\end{align*}
uniformly in $\o \in \Otd$.
\end{thm}
\begin{proof}
It suffices to note that
\begin{align*}
\|\Id\|_{\dual{LM}_{p_2 \,\t, \o}(\rn,v_2) \rw L_{p_1}(\rn,v_1)} &
= \sup_{f \not \sim \,0} \frac{\|f\|_{p_1,v_1,\rn}}{\left\| w(r)
\|f\|_{p_2,v_2,\Brd}
\right\|_{\t,(0,\infty)}} \\
& = \left(\sup_{g \not \sim \,0}
\frac{\|g\|_{q_1,w_1,\rn}}{\left\| H^*(|g|) \right\|_{q_2,w_2,\I}
}\right)^{1 / {p_2}}.
\end{align*}
and apply Theorem \ref{T:5.100}.
\end{proof}

\begin{defi}
Let $X$ be a set of functions from $\M(\rn)$, endowed with a
positively homogeneous functional $\|\cdot \|_X$, defined for
every $f\in \M(\rn)$ and such that $f\in X$ if and only if
$\|f\|_X<\i$. We define the associate space $X'$ of $X$ as the set
of all functions $f\in \M(\rn)$ such that $\|f\|_{X'}<\i$, where
$$
\|f\|_{X'}=\sup \left\{\int_{\rn}|f(x)g(x)|\,dx : \,\,\|g\|_X \leq
1\right\}.
$$
\end{defi}

In \cite{GogMus2} the associate spaces of local Morrey-type and
complementary local Morrey-type spaces were calculated. In
particular, Theorems \ref{main2} and \ref{main200} allows us to
give a characterization of the associate spaces of weighted local
Morrey-type and complementary local Morrey-type spaces.
\begin{thm}\label{thm.6.6}
Assume $1\leq p<\i$, $0<\theta \leq \i$. Let  $\o\in\Ot$ and $v
\in \W(\rn)$. Set
$$
X=\LMv.
$$

\rm{(i)} Let $0<\t\leq 1$. Then
$$
\|f\|_{X'}\ap \sup _{t\in (0,\i)}\|f \|_{p',v^{1 - {p'}},\Btd}
\Vert \o\Vert _{L_{\t}(t,\i)}^{-1},
$$
with the positive constants in equivalence independent of $f$.

\rm{(ii)} Let $1<\t\leq\i$. Then
$$
\|f\|_{X'}\ap \left(\int_{\I} \|f \|_{{p'},v^{1-{p'}},\Btd}^{\t'}d
\left( \|
\o\|_{{\t},(t-,\i)}^{-\t'}\right)\right)^{{1}/{\t'}}+\frac{\|f\|_{{p'},v^{1-{p'}},\rn}}{\|\o\|_{{\t},\I}},
$$
with the positive constants in equivalence independent of $f$.
\end{thm}

\begin{thm}\label{thm.6.5}
Assume $1\leq p<\i$, $0<\theta \leq \i$. Let  $\o\in\Otd$ and $v
\in \W(\rn)$. Set
$$
X=\LMvd.
$$

\rm{(i)} Let $0<\t\leq 1$. Then
$$
\|f\|_{X'}\ap \sup _{t\in (0,\i)}\|f \|_{p',v^{1 - {p'}},\Bt}
\Vert \o\Vert _{L_{\t}(0,t)}^{-1},
$$
with the positive constants in equivalence independent of $f$.

\rm{(ii)} Let $1<\t\leq\i$. Then
$$
\|f\|_{X'}\ap \left(\int_{\I} \|f \|_{{p'},v^{1-{p'}},\Bt}^{\t'}d
\left( \|
\o\|_{{\t},(0,t+)}^{-\t'}\right)\right)^{{1}/{\t'}}+\frac{\|f\|_{{p'},v^{1-{p'}},\rn}}{\|\o\|_{{\t},\I}},
$$
with the positive constants in equivalence independent of $f$.
\end{thm}

\

Rza Mustafayev\\
Department of Mathematics, Faculty of Science and Arts, Kirikkale
University, 71450 Yahsihan, Kirikkale, Turkey\\
E-mail: rzamustafayev@gmail.com\\

Tugce {\"U}nver \\
Department of Mathematics, Faculty of Science and Arts, Kirikkale
University, 71450 Yahsihan, Kirikkale, Turkey\\
E-mail: tugceunver@gmail.com \\


\begin{bibdiv}
\begin{biblist}

\bib{Bur1}{article}{
   author={Burenkov, V.I.},
   title={Recent progress in studying the boundedness of classical operators
   of real analysis in general Morrey-type spaces. I},
   journal={Eurasian Math. J.},
   volume={3},
   date={2012},
   number={3},
   pages={11--32},
   issn={2077-9879},
   review={\MR{3024128}},
}

\bib{BurGol}{article}{
   author={Burenkov, V.I.},
   author={Goldman, M.L.},
   title={Necessary and sufficient conditions for boundedness of the maximal
   operator from Lebesgue spaces to Morrey-type spaces},
   journal={Math. Inequal. Appl. (accepted)},
   volume={},
   date={},
   number={},
   pages={},
}

\bib{BurHus0}{article}{
   author={Burenkov, V.I.},
   author={Guliyev, H.V.},
   title={Necessary and sufficient conditions for the boundedness of the
   maximal operator in local spaces of Morrey type},
   language={Russian},
   journal={Dokl. Akad. Nauk},
   volume={391},
   date={2003},
   number={5},
   pages={591--594},
   issn={0869-5652},
   review={\MR{2042888 (2004k:42030)}},
}

\bib{BurHus1}{article}{
   author={Burenkov, V. I.},
   author={Guliyev, H. V.},
   title={Necessary and sufficient conditions for boundedness of the maximal
   operator in local Morrey-type spaces},
   journal={Studia Math.},
   volume={163},
   date={2004},
   number={2},
   pages={157--176},
   issn={0039-3223},
   review={\MR{2047377 (2005c:42018)}},
 }

\bib{BurGulHus1}{article}{
   author={Burenkov, V.I.},
   author={Guliyev, H.V.},
   author={Guliyev, V.S.},
   title={On boundedness of the fractional maximal operator from
   complementary Morrey-type spaces to Morrey-type spaces},
   conference={
      title={The interaction of analysis and geometry},
   },
   book={
      series={Contemp. Math.},
      volume={424},
      publisher={Amer. Math. Soc.},
      place={Providence, RI},
   },
   date={2007},
   pages={17--32},
   review={\MR{2316329 (2008d:42004)}},
}

\bib{BurGulSerTar}{article}{
   author={Burenkov, V.I.},
   author={Guliyev, V.S.},
   author={Serbetci, A.},
   author={Tararykova, T.V.},
   title={Necessary and sufficient conditions for the boundedness of genuine
   singular integral operators in local Morrey-type spaces},
   journal={Eurasian Math. J.},
   volume={1},
   date={2010},
   number={1},
   pages={32--53},
   issn={2077-9879},
   review={\MR{2898674}},
}

\bib{BGGM}{article}{
   author={Burenkov, V.I.},
   author={Gogatishvili, A.},
   author={Guliyev, V.S.},
   author={Mustafayev, R.Ch.},
   title={Boundedness of the fractional maximal operator in local
   Morrey-type spaces},
   journal={Complex Var. Elliptic Equ.},
   volume={55},
   date={2010},
   number={8-10},
   pages={739--758},
   issn={1747-6933},
   review={\MR{2674862 (2011f:42015)}},
}

\bib{BGGM1}{article}{
   author={Burenkov, V.I.},
   author={Gogatishvili, A.,}
   author={Guliyev, V.S.},
   author={Mustafayev, R.Ch.},
   title={Boundedness of the Riesz potential in local Morrey-type spaces},
   journal={Potential Anal.},
   volume={35},
   date={2011},
   number={1},
   pages={67--87},
   issn={0926-2601},
   review={\MR{2804553 (2012d:42027)}},
}

\bib{BurJainTar}{article}{
   author={Burenkov, V.I.},
   author={Jain, P.},
   author={Tararykova, T.V.},
   title={On boundedness of the Hardy operator in Morrey-type spaces},
   journal={Eurasian Math. J.},
   volume={2},
   date={2011},
   number={1},
   pages={52--80},
   issn={2077-9879},
   review={\MR{2910821}},
}

\bib{BurNur}{article}{
   author={Burenkov, V.I.},
   author={Nursultanov, E.D.},
   title={Description of interpolation spaces for local Morrey-type spaces},
   language={Russian, with Russian summary},
   journal={Tr. Mat. Inst. Steklova},
   volume={269},
   date={2010},
   number={Teoriya Funktsii i Differentsialnye Uravneniya},
   pages={52--62},
   issn={0371-9685},
   translation={
      journal={Proc. Steklov Inst. Math.},
      volume={269},
      date={2010},
      number={1},
      pages={46--56},
      issn={0081-5438},
   },
   review={\MR{2729972 (2011g:46034)}},
}

\bib{ChristGraf}{article}{
   author={Christ, M.},
   author={Grafakos, L.},
   title={Best constants for two nonconvolution inequalities},
   journal={Proc. Amer. Math. Soc.},
   volume={123},
   date={1995},
   number={6},
   pages={1687--1693},
   issn={0002-9939},
   review={\MR{1239796 (95g:42031)}},
}

\bib{DrabHeinKuf}{article}{
   author={Dr{\'a}bek, P.},
   author={Heinig, H.P.},
   author={Kufner, A.},
   title={Higher-dimensional Hardy inequality},
   conference={
      title={General inequalities, 7},
      address={Oberwolfach},
      date={1995},
   },
   book={
      series={Internat. Ser. Numer. Math.},
      volume={123},
      publisher={Birkh\"auser},
      place={Basel},
   },
   date={1997},
   pages={3--16},
   review={\MR{1457264 (98k:26026)}},
}

\bib{GogMus1}{article}{
   author={Gogatishvili, A.},
   author={Mustafayev, R.Ch.},
   title={Dual spaces of local Morrey-type spaces},
   journal={Czechoslovak Math. J.},
   volume={61(136)},
   date={2011},
   number={3},
   pages={609--622},
   issn={0011-4642},
   review={\MR{2853078}},
 }

\bib{GogMus2}{article}{
   author={Gogatishvili, A.},
   author={Mustafayev, R. Ch.},
   title={The multidimensional reverse Hardy inequalities},
   journal={Math. Inequal. Appl.},
   volume={15},
   date={2012},
   number={1},
   pages={1--14},
   issn={1331-4343},
   review={\MR{2919426}},
   doi={10.7153/mia-15-01},
}

\bib{GogMus3}{article}{
   author={Gogatishvili, A.},
   author={Mustafayev, R.Ch.},
   title={New pre-dual space of Morrey space},
   journal={J. Math. Anal. Appl.},
   volume={397},
   date={2013},
   number={2},
   pages={678--692},
   issn={0022-247X},
   review={\MR{2979604}},
}

\bib{Gul}{book}{
   author={Guliyev, V.S.},
   title={Function spaces, integral operators and two weighted inequalities on homogeneous groups. Some applications},
   publisher={Elm},
   place={Baku},
   language={Russian},
   date={1999},
}

\bib{GulMus0}{article}{
   author={Guliev, V.S.},
   author={Mustafaev, R.Ch.},
   title={Integral operators of potential type in spaces of homogeneous
   type},
   language={Russian},
   journal={Dokl. Akad. Nauk},
   volume={354},
   date={1997},
   number={6},
   pages={730--732},
   issn={0869-5652},
   review={\MR{1473130}},
}

\bib{GulMus1}{article}{
   author={Guliev, V. S.},
   author={Mustafaev, R. Ch.},
   title={Fractional integrals in spaces of functions defined on spaces of
   homogeneous type},
   language={Russian, with English and Russian summaries},
   journal={Anal. Math.},
   volume={24},
   date={1998},
   number={3},
   pages={181--200},
   issn={0133-3852},
   review={\MR{1639211 (99m:26007)}},
}

\bib{kp}{book}{
   author={Kufner, A.},
   author={Persson, L.-E.},
   title={Weighted inequalities of Hardy type},
   publisher={World Scientific Publishing Co. Inc.},
   place={River Edge, NJ},
   date={2003},
   pages={xviii+357},
   isbn={981-238-195-3},
   review={\MR{1982932 (2004c:42034)}},
}

\bib{Mucken}{article}{
   author={Muckenhoupt, B.},
   title={Weighted norm inequalities for the Hardy maximal function},
   journal={Trans. Amer. Math. Soc.},
   volume={165},
   date={1972},
   pages={207--226},
   issn={0002-9947},
   review={\MR{0293384 (45 \#2461)}},
}

\bib{ok}{book}{
   author={Opic, B.},
   author={Kufner, A.},
   title={Hardy-type inequalities},
   series={Pitman Research Notes in Mathematics Series},
   volume={219},
   publisher={Longman Scientific \& Technical},
   place={Harlow},
   date={1990},
   pages={xii+333},
   isbn={0-582-05198-3},
   review={\MR{1069756 (92b:26028)}},
}

\bib{r}{book}{
   author={Rudin, W.},
   title={Principles of mathematical analysis},
   series={Second edition},
   publisher={McGraw-Hill Book Co.},
   place={New York},
   date={1964},
   pages={ix+270},
   review={\MR{0166310 (29 \#3587)}},
}


\end{biblist}
\end{bibdiv}
\end{document}